\documentclass[a4paper]{amsart}
\usepackage{amsfonts}
\usepackage{amsmath}
\usepackage{mathrsfs}
\usepackage{graphicx}
\usepackage{amsmath,amsthm,amssymb,amscd}
\usepackage{color}
\usepackage{enumerate}
\usepackage{epsfig}
\usepackage{graphicx}
\usepackage{tikz}
\usepackage{tikz-cd}
\usepackage[T1]{fontenc}
\usepackage[all]{xy}
\usepackage[english]{babel}
\usepackage{mathtools}
\newtheorem{thm}{Theorem}[section]

\theoremstyle{remark}
\newtheorem{rem} [thm]{Remark}

\theoremstyle{definition}
\newtheorem{clm}[thm]{Claim}
\newtheorem{que}[thm]{Question}

\DeclareMathOperator{\diam}{diam}

\newcommand{\N}{{\mathbb{N}}}
\newcommand{\Z}{\mathbb{Z}}

\newcommand{\tg}{\tilde{g}}

\begin{document}

\title[On Rigid Minimal Spaces]{On Rigid Minimal Spaces}
\author[J. Boro\'nski]{Jan P. Boro\'nski}
\address[J. Boro\'nski]{AGH University of Science and Technology, Faculty of Applied Mathematics, al. Mickiewicza 30, 30-059 Krak\'ow, Poland. -- and -- National Supercomputing Centre IT4Innovations, Division of the University of Ostrava, Institute for Research and Applications of Fuzzy Modeling, 30. dubna 22, 70103 Ostrava, Czech Republic}
\email{ boronski@agh.edu.pl}

\author[J.  \v Cin\v c]{Jernej  \v Cin\v c}
\address[J. \v Cin\v c]{AGH University of Science and Technology, Faculty of Applied Mathematics, al. Mickiewicza 30, 30-059 Krak\'ow, Poland -- and -- National Supercomputing Centre IT4Innovations, Division of the University of Ostrava, Institute for Research and Applications of Fuzzy Modeling, 30. dubna 22, 70103 Ostrava, Czech Republic}
\email{jernej.cinc@osu.cz}

\author[M. Fory\'{s}-Krawiec]{Magdalena Fory\'{s}-Krawiec}
\address[M. Fory\'{s}-Krawiec]{AGH University of Science and Technology, Faculty of Applied Mathematics, al. Mickiewicza 30, 30-059 Krak\'ow, Poland -- and -- National Supercomputing Centre IT4Innovations, Division of the University of Ostrava, Institute for Research and Applications of Fuzzy Modeling, 30. dubna 22, 70103 Ostrava, Czech Republic}
\email{maforys@agh.edu.pl}

\begin{abstract}
A compact space $X$ is said to be minimal if there exists a map $f:X\to X$ such that the forward orbit of any point is dense in $X$. We consider rigid minimal spaces, motivated by recent results of Downarowicz, Snoha, and Tywoniuk  [{\it J. Dyn. Diff. Eq.}, 2016] on spaces with cyclic group of homeomorphisms generated by a minimal homeomorphism, and results of the first author, Clark and Oprocha [{\em Adv. Math.}, 2018] on spaces in which the square of every homeomorphism is a power of the same minimal homeomorphism. We show that the two classes do not coincide, which gives rise to a new class of spaces that admit minimal homeomorphisms, but no minimal maps. We modify the latter class of examples to show for the first time existence of minimal spaces with degenerate homeomorphism groups. Finally, we give a method of constructing decomposable compact and connected spaces with cyclic group of homeomorphisms, generated by a minimal homeomorphism, answering a question in Downarowicz et al.
\end{abstract}

\maketitle
\section{introduction}
Given a compact space $X$, a map $f:X\to X$ is said to be \emph{minimal} if for every $x\in X$ the forward orbit $\{f^n(x):n\in\mathbb{N}\}$\footnote{$\mathbb{N}=\{0,1,2,\ldots\}$} is dense in $X$. A compact space is said to be minimal if it admits a minimal map. The best known examples of minimal spaces are the Cantor set, circle $\mathbb{S}^1$ and 2-torus $\mathbb{T}^2$. The classification of minimal spaces is a well-known open problem in topological dynamics, motivated by the fact that minimal dynamical systems are building blocks for all dynamical systems \cite{Birkhoff}. Even for manifolds the answer is unkown in dimension greater than $2$.  This also applies to $\mathbb{R}$-flows, with the still unresolved Gottschalk Conjecture: {\it $\mathbb{S}^3$ does not admit a minimal $\mathbb{R}$-flow}. The class of minimal maps contains both invertible and noninvertible elements, and a space may admit both minimal homeomorphisms and minimal noninvertible maps, either one or the other, or none. Auslander and Yorke \cite{AY} showed that the Cantor set admits both types of minimal maps, whereas Auslander and Katznelson \cite{AK} showed that $\mathbb{S}^1$ does not admit minimal noninvertible maps. Katok \cite{Katok}, and Fathi and Herman \cite{FH} showed that every compact connected manifold, which admits a smooth locally free effective action of the circle group, has a smooth minimal diffeomorphism, isotopic to the identity, and so all odd dimensional spheres admit a minimal diffeomorphism. Every manifold that admits a minimal flow admits also minimal homeomorphisms (see e.g. \cite{F}), as well as minimal noninvertible maps \cite{BCO}. Kozlowski and the first author \cite{BKoz} proved that if a~compact manifold of dimension at least $2$ admits a minimal homeomorphism, then it admits a minimal noninvertible map. Minimal sets for surface homeomorphisms were classified in \cite{JKP} and \cite{PX}. In \cite{KST} Kolyada, Snoha and Trofimchuk constructed minimal noninvertible maps on $\mathbb{T}^2$ (see also \cite{Sotola}). Their result was modified by Bruin, Kolyada and Snoha, who constructed in \cite{BKS} an example of a minimal $2$-dimensional space $X$, that has the fixed point property for homeomorphisms. 

This paper concerns spaces that have very simple homeomorphism groups, and yet they admit minimal maps. This direction of research was motivated by recent results of Downarowicz, Snoha, and Tywoniuk in \cite{DST}, and the first author, Clark, and Oprocha in \cite{BCO}. The authors of \cite{DST} constructed the first examples of minimal spaces with cyclic homeomorphism groups, generated by a minimal homeomorphism (called {\it Slovak spaces} in \cite{DST}). Moreover, they showed that some of such spaces do not admit minimal noninvertible maps - the first examples of minimal spaces without minimal noninvertible maps other than the circle. In \cite{BCO} the authors used an inverse limit approach to construct examples of a space $X$, such that the homeomorphism group $$\mathrm{H}(X)=\mathrm{H}_+(X)\cup \mathrm{H}_{-}(X),$$
with
$$\mathrm{H}_{+}(X)\cap \mathrm{H}_{-}(X)=\{\operatorname{id}_X\},$$
where $\mathrm{H}_+(X)$ is cyclic and generated by a minimal homeomorphism, and for every $g\in \mathrm{H}_{-}(X)$ we have $g^2\in \mathrm{H}_{+}(X)$ (in \cite{BCO} called {\it almost Slovak spaces}, by reference to \cite{DST}). These examples were used in \cite{BCO} to give a negative answer to the following problem: given two compact spaces $X$ and $Y$ admitting minimal homeomorphism, does the product $X\times Y$ admit a minimal homeomorphism as well? In \cite{SS} the authors improve on  this result to show that the product of two compact spaces admitting minimal homeomorphisms does not need to admit any minimal maps. Slovak and almost Slovak spaces give rise to a kind of paradox: {\it the spaces admit homeomorphisms under which all points travel almost everywhere, but at the same time they are so rigid that they admit very few homeomorphisms}. In the present paper, by exploiting the methods from \cite{BCO}, in Section \ref{section:degenerate} we push this paradox even further, by constructing first examples of minimal spaces with degenerate homeomorphism groups. 
\vspace{0.2cm}
\newline
{\bfseries Theorem \ref{thm:degenerate}. }{\it
	There exist uncountably many minimal spaces with degenerate homeomorphism groups. In addition, for any real number $r\geq 0$ there exists such a~space that admits a minimal map with topological entropy $r$. }
\vspace{0.2cm}

\noindent
Earlier, in Section \ref{sec:AlmostSlovakNotSlovak} we show for the first time examples of almost Slovak spaces that are not Slovak, and then we show that these spaces do not admit minimal noninvertible maps. This gives a new family of examples of minimal spaces without minimal noninvertible maps, and the first class that is neither $\mathbb{S}^1$ nor a Slovak space. 
\vspace{0.2cm}

\noindent
{\bfseries Theorem \ref{cor}. }{\it There exist minimal spaces without minimal noninvertible maps, that are neither $\mathbb{S}^1$ nor Slovak. }
\vspace{0.2cm}

\noindent
In Section \ref{section:decomposable}, answering Question 2 from \cite{DST} we show the existence of Slovak spaces that are decomposable; i.e. they decompose to a union of two distinct proper nondegenerate connected and compact subsets. This contrasts with indecomposability of the previously known examples in \cite{BCO} and \cite{DST}. 
\vspace{0.2cm}

\noindent
{\bfseries Theorem \ref{thm:decomposable}. }{\it
		Each of the following classes of spaces contains an uncountable family of decomposable continua.
	\begin{itemize}
		\item[(D1)] Slovak spaces,
		\item[(D2)] almost Slovak spaces that are not Slovak,
		\item[(D3)] minimal spaces with degenerate homeomorphism groups.
\end{itemize} }
\vspace{0.000001cm}

\noindent
We conclude this introduction with the following question that arises naturally from our study.
\begin{que}\label{q1}
For what groups $G$ does there exist a minimal connected and compact metric space $X$, such that the homeomorphism group $\mathrm{H}(X)$ is isomorphic to $G$? When can $X$ be chosen to be decomposable, or locally connected?
\end{que}
\noindent
Note that de Groot showed in \cite{deGroot} that every group can be represented as the group of isometries of a connected, locally connected, complete metric space, and if the group is countable then the space can be chosen to be a locally connected metric continuum \cite{GW}.
\section{Preliminaries}

\subsection{Basic notions and definitions}
Let $X$ be a compact metric space. A continuous map $T:X\rightarrow X$ is \textit{minimal} if there is no proper subset $M\subset X$ which is nonempty, closed and $T$-invariant. In other words, $T$ is minimal if the forward orbit $\{T^n(x): n \in \mathbb{N}\}$ is dense in $X$ for every point $x \in X$. By $\mathcal{O}(x) = \{T^i(x): i \in \mathbb{Z}\} $ we denote the orbit (backward and forward) of a point $x \in X$. A \textit{continuum} is a~compact connected nondegenerate metric space. A continuum is \textit{decomposable} if it is the union of two distinct proper subcontinua. A continuum which is not decomposable is said to be \textit{indecomposable}. A continuum is \textit{hereditarily indecomposable} if every subcontinuum is indecomposable. The \textit{composant} of a point $x \in X$ is the union of all proper subcontinua of $X$ containing $x$. A \textit{decomposition} $\mathcal{D} \subset \mathcal{P}(X)$ of a space $X$ is a partition of $X$ to pairwise disjoint nonempty sets. Decomposition $\mathcal{D}$ is an \textit{upper semicontinuous decomposition} if the elements of $\mathcal{D}$ are compact and for every $D \in \mathcal{D}$ and every open set $U \subset X$ containing $D$ there exists another open set $V\subset X$ containing $D$ such that for every $D'\in \mathcal{D}$ if $D'\cap V\neq \emptyset$ then $D'\subset U$. Let $\pi:X\to Y$ be a continuous surjection between two compact metric spaces and let $Y_{\star}$ denote the set of points $y\in Y$ whose fibers $\pi^{-1}(y)$ are singletons. Let $X_{\star}=\pi^{-1}(Y_{\star})=\{x\in X: \pi^{-1}(\pi(x))=x\}$. If $X_\star$ is dense in $X$, then $\pi$ is called \textit{almost one-to-one}. If $\{X_i\}_{i \in \mathbb{N}}$ is a sequence of topological spaces and $\{f_i\}_{i \in \mathbb{N}}$ is a sequence of mappings with $f_i:X_{i+1}\to X_i$ then the \textit{inverse limit} $\varprojlim (X_i,f_i)$ is a subset of $\prod_{i>0}X_i$ such that $x \in \varprojlim (X_i,f_i)$ if and only if $f_i(x_{i+1})=x_i$ for $i \in \mathbb{N}$. If $f_i=f$ and $X_i=X$ for every $i\in \N$ we define the \emph{natural extension} $\sigma_{f}:\varprojlim (X,f)\to \varprojlim (X,f)$ (or sometimes called \emph{shift homeomorphism}) by $\sigma_{f}(x_1,x_2,x_3,\ldots)=(f(x_1),x_1,x_2,\ldots)$. If a space can be represented as the inverse limit of arcs, then it is said to be  \textit{arc-like}. The \textit{pseudo-arc} is a topologically unique arc-like hereditarily indecomposable space \cite{B}. Together with the Cantor set, $\mathbb{S}^1$ and all finite sets it forms a family of building blocks for all topologically homogeneous planar compacta \cite{HO}. A nondegenerate compact space $X$ is said to be a \textit{Slovak space} if its group of homeomorphisms is of the form $H(X) = \{ T^n: n \in \mathbb{Z} \}$, where $T$ is minimal. A~compact metric space $X$ is an \textit{almost Slovak space} if its homeomorphism group has the following property:
$$
H(X) = H_+(X)\cup H_-(X),
$$
with 
$$
H_+(X)\cap H_-(X) = \{id_X\} 
$$
where $H_+(X)$ is cyclic and generated by a minimal homeomorphism, and for every $g \in H_-(X)$ we have $g^2 \in H_+(X)$.


\subsection{Almost Slovak spaces from inverse limits}\label{subsec:AlmostSlovak}
We start by recalling the main points of the construction from \cite{BCO}.  In \cite[Theorem 3.1]{BCO} the authors take a minimal homeomorphism $(C,h)$ of the Cantor set $C$, the minimal suspension $(X,f)$ of $(C,h)$ guaranteed by \cite{F} and use the methodology from \cite[Theorem 2.3]{BCO} to construct an almost Slovak space $\mathbb{Y}$ as an inverse limit of some properly constructed spaces $(Y_n, \gamma_n)$. The method used in the proof of Theorem 2.3 requires applying some perturbations on $(X,f)$. The points from the orbit $\mathcal{O}(q)$ of a chosen point $q \in X$ is replaced with a null sequence of arcs to get the almost one-to-one extension $(X,\widetilde{H})$ of the suspension $(X,f)$. A~continuous map $g:[-1,2]\rightarrow [-1,2]$ is defined as identity outside the interval $(0,1)$ such that its inverse limit on $[0,1]$ is the pseudo-arc. 
The function is then used while defining $Y_1 = \varprojlim(X,g_n)_{n=1}^{\infty}$, where the bonding functions are dependent on $g$ on some properly chosen open set. We may see $Y_1$ as $X$  with the point $q$ removed, and the resulting "hole" compactified by a pseudo-arc. The consecutive spaces $Y_n$ for $n\geq 2$ are defined analogously, that is in the $n$-th step we remove the points $\{f^k(q):k=-n,-n+1,....,n-1,n\} $ and compactify each hole in the space by a pseudo-arc. The mapping $\gamma_n:Y_{n+1}\rightarrow Y_n$ is a natural projection, collapsing two pseudo-arcs to two points for every $n \in \mathbb{N}$. Those are used to define $\mathbb{Y} = \varprojlim(Y_n,\gamma_n)$. Observe that the natural projection $\pi:\mathbb{Y}\rightarrow X$ is one-to-one onto every point but the orbit of $q$, and $\pi^{-1}(f^i(q))$ is the pseudo-arc for every $i \in \mathbb{Z}$. In the last part of the proof the authors construct a~minimal homeomorphism $H_{\mathbb{Y}}:\mathbb{Y}\rightarrow \mathbb{Y}$ as follows:
$$
H_{\mathbb{Y}}(y) = (\pi^{-1}\circ F \circ \pi)(y) \text{ if } \pi^{-1}(y)\notin \mathcal{O}(q)
$$
and extend the above definition on the points with $\pi^{-1}(y) \in \mathcal{O}(q)$. As a consequence of their construction it follows that all non-singleton fibers $\pi^{-1}(y)$ are pseudo-arcs and all those pseudo-arcs are contained in exactly one composant of $\mathbb{Y}$ which we denote by $W_{\mathbb{Y}}$.  This method can be also used to construct Slovak spaces, for example by choosing $g$ so that the inverse limit with $g|_{[0,1]}$ is homeomorphic to $\sin(1/x)$-curve. 


\section{Almost Slovak spaces that do not admit minimal noninvertible maps}\label{sec:AlmostSlovakNotSlovak}
In this section we shall adapt the approach from the proof of \cite[Theorem 3.1]{BCO}, to construct an almost Slovak space $Y$, which is not Slovak, and then we will show that they do not admit minimal noninvertible maps. Our aim is to construct orientation reversing homeomorphisms on the spaces from \cite{BCO}, and since any such homeomorphism must have a fixed point, we thus obtain almost Slovak spaces which are not Slovak.  
Recall that the suspension flow $(\Phi_t)_{t \in \mathbb{R}}$ over a homeomorphism $h~:~X\to X$ of a~compact space $X$ is the flow defined on the space $X\times [0,1]/_{(y,1)\sim (h(y),0)}$ as follows:
$$
\Phi_t(y,s) = (h^{\lfloor t+s \rfloor}(y), \{t+s\}),
$$
where $\lfloor x \rfloor$ and $\{ x \}$ denote the integer and the fractional part of $x$ respectively.  

\begin{thm}\label{A}
	There exists an almost Slovak space $Y$ that is not Slovak. 	
\end{thm}
\begin{proof}
	Let $X$ be the $2$-adic solenoid. Let $(\Phi_t)_{t \in \mathbb{R}}$ be the suspension flow on $X$ over the $2$-adic odometer $h$. Let us represent the $2$-adic solenoid as $X = (\Lambda_2\times \mathbb{R})/\approx$, where $\Lambda_2$ is the $2$-adic Cantor set, seen as the group of $2$-adic integers, and $(c,x)\approx(c',x')$ if and only if $h(c)=c'$ and $x+1=x'$. By \cite{Kw} $X$ admits an orientation reversing algebraic homeomorphism $A$ induced by $A': \Lambda_2\times \mathbb{R}\to \Lambda_2\times \mathbb{R}$ given by 
	$$
	A'(c,x) = (c^{-1},-x).
	$$
	Note that $A^2=\operatorname{id}_X$ and $A$ fixes the arc-component of $q=(1,0)$. Choose $\alpha\notin\mathbb{Q}$ such that $10<\alpha<11$. Then $\Phi_{\alpha}$ is minimal \cite{F}. We are going to show that there exists a continuum $Y$, a minimal homeomorphism $H:Y\to Y$, a surjection $\pi:Y\to X$, and a nonminimal orientation reversing homeomorphism $a:Y\to Y$ such that:
	\begin{enumerate}[(i)]
		\item $H=\pi^{-1}\circ \Phi_\alpha\circ\pi$,
		\item $a=\pi^{-1}\circ A\circ\pi$, 
		\item $\pi$ is almost one-to-one,
		\item all non-singleton fibers $\pi^{-1}(q)$ are pseudo-arcs,
		\item there exists a composant $W\subset Y$ such that if $|\pi^{-1}(x)|>1$ then $\pi^{-1}(x)\subset W$.
		\item $\lim_{|i|\to \infty}\diam H^i\left(\pi^{-1}(x)\right)=0$ for all $x$.
	\end{enumerate}
	 First, by the proof of \cite[Theorem 2.3]{BCO} (see \cite[Remark 2.4]{BCO}) we obtain a space $X_\infty$, that is homeomorphic to $X$, and such that $\Phi_\alpha$ lifts to a minimal homeomorphism~ $F:X_\infty\to X_\infty$. The homeomorphism $F$ is semi-conjugate to $\Phi_\alpha$, and has a property that there exists and arc $I\subset X_\infty$ such that $\{F^k(I)\}_{k\in\mathbb{Z}}$ forms a null sequence. The semi-conjugacy is one-to-one outside of $\bigcup_{k\in\mathbb{Z}} \{F^k(I)\}$. We may assume that $X_\infty=X$ and $I=\{1\}\times [-2,2]$. Let $F^k(I)=\{1\}\times I_k$ for all $k\in\mathbb{Z}$. Without loss of generality we may assume that $I_k=[k\alpha-\frac{2}{1+k^2},k\alpha+\frac{2}{1+k^2}]$. 
	There exists a nested sequence of clopen sets $C_i \subset \Lambda_2$ such that $\bigcap_i C_i=\{1\}$.	Let $\tg:[-1,1]\rightarrow [-1,1]$ be a piecewise linear map determined by $\tg(-\frac{1}{3})=\tg(1)=-1$ and $\tg(\frac{1}{3})=\tg(-1)=1$. By \cite[Theorem 14]{KOT} we have a map $g: [-1,1]\rightarrow [-1,1]$ such that $g(x)=\tg(x)$ for all $x\in\{-1,-\frac{1}{3},\frac{1}{3},1\}$, and  such that $\varprojlim(g,[-1,1])$ is the pseudo-arc. Note that $\varprojlim(g^2,[-1,1])$ is also a pseudo-arc, since $\varprojlim(g,[-1,1])\simeq \varprojlim(g^2,[-1,1])$ (from now on we use $\simeq$ to denote homeomorphic spaces); see e.g. \cite[Corollary 2.5.11]{Engelking}. Extend $g$ to a~continuous surjection on $[-2,2]$, by putting $g(x)=-x$ for all $x\notin (-1,1)$. Now for each $n \in \mathbb{N}$ let $f_n^{(0)}\colon X\to X$ be a map such that:
	\begin{eqnarray*}
		f_n^{(0)}(c,x) &=& (c^{-1},g(x)) \text{ for } x\in I_0, c \in C_n, \\
		f_n^{(0)}(c,x) &=& A(c,x) \text{ otherwise. } 
	\end{eqnarray*}
	Let $X_0=\varprojlim ((X,f_n^{(0)})_{n=1}^\infty)$. Observe that for each $x\not\in I_0$ there exists a natural number $N$ and an open neighborhood $V$ of $(1,x)$, such that $f_n^{(0)}(c,x')=A(c,x')$ for all $(c,x')\in V$ and $n>N$. But since a finite number of intial coordinates does not affect the topological structure of an inverse limit space, we see that if $(x_1,x_2,x_3,\ldots)\in X_0$ and $x_j\not \in \{1\}\times I_0$ for every $j\in \N$, then a small neighborhood of $x$ is homeomorphic to $\Lambda_2\times (0,1)$. If, on the other hand, $x_j\in I_0$ for some (thus all) $j\in\N$
	then $x\in \varprojlim ((I_0,f_n^{(0)})_{n=1}^\infty)\simeq \varprojlim (I_0,g)=P_0$, which is a unique maximal pseudo-arc embedded in $X_0$. In addition, by the above we get a homeomorphism $a_0:X_0\to X_0$ which is semi-conjugate to $A$, such that $a_0|_{(X_0\setminus P_0)}$ is conjugate to $A|_{(X\setminus \{1\}\times I_0)}$ and $a_0|_{P_0}$ is conjugate to the natural extension $\sigma_g$. 
	We may view $X_0$ as $X$ with $I_0$ replaced by $P_0$. 
	Let  $\xi_0\colon X_0\to X$ be the projection onto the first coordinate of the inverse limit space $ \varprojlim ((X,f_n^{(0)})_{n=1}^\infty)$.  Note that $\xi_0$ is one-to-one, except on $P_0\subset X_0$, with  $\xi_0(P_0)=I_0$. 
	
	Now suppose we have already defined the continua $X_m=\varprojlim ((X,f_n^{(m)})_{n=1}^\infty)$, the finite family of pseudo-arcs $\mathcal{P}_m=\{P_j:j=-m,-m+1,\ldots, m-1,m\}$ and projections $\xi_m:X_m\to X$, which are one-to-one except on $\bigcup \mathcal{P}_m$, such that $\xi_m(P_j)=I_j$, for all $j\in \{-m,-m+1,\ldots, m-1,m\}$, and $\diam (P_j)=I_j$. For every positive integer $m$ we have a natural projection $$\gamma_{m-1} \colon X_{m}\to X_{m-1}$$ which maps pseudo-arcs $P_{-m},P_{m}$ onto arcs $\xi^{-1}_{m}(I_{-m})$ and $\xi^{-1}_{m}(I_{m})$ respectively. We also have a family of orientation reversing homeomorphisms $(a_i:i=0,\ldots,m)$ such that for all $i=0,\ldots,m$ the following properties are satisfied:
	\begin{itemize}
	\item $a_i:X_i\to X_i$ is semi-conjugate to $A$ via $\xi_i$,
	\item $a_j:X_j\to X_j$ is semi-conjugate to $a_{j-1}$ via $\gamma_j$ for all $0< j\leq i$,
	\item $a_i|_{(X_i\setminus \bigcup\mathcal{P}_i)}$ is conjugate to $A|_{(X\setminus \bigcup_{j=-i}^{i}(\{1\}\times I_i))}$, and 
	\item $a_i|_{\bigcup\mathcal{P}_i}$ is conjugate to $\sigma_g$. 
	\end{itemize}
	To define $X_{m+1}$ we define $f^{(m+1)}_n\colon X_m\to X_m$ as:
		$$ f_n^{(m+1)}(c,x) = (c^{-1},F^{-j}\circ g\circ F^{-j}(x)),$$ for  $x\in I_j, c \in F^j(C_n), j=-m-1,\ldots,m+1,$ and
		$$f_n^{(m+1)}(c,x) = A(c,x),$$ otherwise. We define $X_{m+1}=\varprojlim ((X,f_n^{(m+1)})_{n=1}^\infty)$ and $P_{j}=\varprojlim ((I_j,f_n^{(m+1)})_{n=1}^\infty)$ for $j=-m-1,\ldots, m+1$. 
	 Furthermore observe that we have a natural projection $$\gamma_m \colon X_{m+1}\to X_m$$ which maps pseudo-arcs $P_{-m-1},P_{m+1}$ onto arcs $\xi^{-1}_{m+1}(I_{-m-1})$ and $\xi^{-1}_{m+1}(I_{m+1})$ respectively. Let $Y=\varprojlim ((X_{m-1},\gamma_m)_{m=0}^\infty)$, where $X_{-1}=X$, and let $\pi \colon Y\to X$ be the projection from the inverse limit space onto the coordinate space $X_{-1}$. Observe that $\pi^{-1}(\{1\}\times I_k)$ is a pseudo-arc for every $k\in\mathbb{Z}$ and $\pi$ is one-to-one on $Y\setminus\bigcup_{k\in\mathbb{Z}} \pi^{-1}(\{1\}\times I_k)$. All composants of $Y$ are continuous one-to-one images of the real line, except the composant $W=\pi^{-1}(\mathcal{C}_{(1,0)})$, where $\mathcal{C}_{(1,0)}$ is the composant of $(1,0)$ in $X$, which contains countably many pseudo-arcs connected by arcs. We define the homeomorphism $a:Y\to Y$ by $a=(a_0,a_1,a_2,\ldots)$. Homeomorphism $a$ is well defined since the homeomorphisms $(a_m:m=0,1,\ldots)$ commute with the bonding maps $(\gamma_m:m=0,1,\ldots)$.  Note that $a=\pi^{-1}\circ A\circ\pi$. It follows from the construction that $a^2=\mathrm{id}_Y$ since $a^2$ agrees with $\mathrm{id}_Y$ on a dense set by the fact that $A^2=\mathrm{id}_X$. Now we argue that the homeomorphism $a:Y\to Y$ is nonminimal. But this is easy to see, since $a(\pi^{-1}(\{1\}\times I_0))=\pi^{-1}(\{1\}\times I_0)$ and because $\pi^{-1}(\{1\}\times I_0)$ is a pseudo-arc that has the fixed point property \cite{Hamilton} it follows that $a$ has a fixed point.  
	 
	 Next we define the homeomorphism $H\colon Y\to Y$. Since we have
	 $$X_m=\varprojlim ((X,f_n^{(m)})_{n=1}^\infty)\simeq \varprojlim ((X,f_{n+1}^{(m)}\circ f_n^{(m)})_{n=1}^\infty),$$
	 and 
	 $$f_{n+1}^{(m)}\circ f_n^{(m)}(x)=x$$ for every $x$ outside of a small neighborhood of $\bigcup_{j=-m}^{m} (\{1\}\times I_j)$, and $$f_{n+1}^{(m)}\circ f_n^{(m)}(x)=g^2(x)$$ for $x\in\bigcup_{j=-m}^{m} (\{1\}\times I_j)$ we argue analogously as in the proof of \cite[Theorem 3.1]{BCO} to get a minimal homomorphism $H$ that is an almost one-to-one extension of $F$, and the proof is complete. 	 
	\end {proof}
	
Now we shall show that the above almost Slovak spaces, in fact all constructed in \cite[Theorem 3.1]{BCO}, do not admit minimal noninvertible maps. 
\begin{thm}\label{hom}
	Every minimal map on $Y$ is a homeomorphism.
\end{thm}
\begin{proof}
Let $F:Y\to Y$ be the generator of the homeomorphism group $H(Y)$ and $G:Y\to Y$ be a continuous minimal noninvertible surjection, fixed once and for all. We start with the following straightforward observations, which we include for completness. 

\begin{clm}\label{cl:comp-comp}
	Assume $K$ is a composant of $Y$. Then either $G(K)=Y$ or $G(K)$ is a~subset of a~composant of $Y$.
\end{clm} 
\begin{proof}(of Claim \ref{cl:comp-comp})
	Assume that $G(K)\neq Y$ and $K$ is mapped by $G$ to two different composants $K_1\neq K_2$. Since every composant in $Y$ is connected and $G$ is continuous it holds that $K_1 \cap K_2 \neq \emptyset$. However, since $Y$ is an indecomposable continuum, it consists of pairwise disjoint composants, therefore $K_1=K_2$. 
\end{proof}

\begin{clm}\label{clm:f(W)=W}
 Let $W$ be the special composant of $Y$, that contains a null sequence of pseudo-arcs. Then $G^{-1}(W)=W=G(W)$.
\end{clm}
\begin{proof}(of Claim~\ref{clm:f(W)=W})
	Let $\mathcal{K}$ denote the family of pathwise connected composants of $Y$. It consists of all composants of $Y$ but $W$. Each composant in $\mathcal{K}$ is dense in $Y$ while path components of $W$ are not dense. $G(K)\neq Y$ for any $K\in\mathcal{K}$ since $Y$ is not pathwise connected. Also $G(W)\neq Y$ since $W$ has no dense path components. Hence for any $K \in \mathcal{K}$ we have $G(K) \in \mathcal{K}$ which implies that $G^{-1}(W)=W$ and, by surjectivity, $G(W)=W$.
\end{proof}
By $(I_n)_{n \in \mathbb{Z}}$ denote the path components of $W$, keeping in mind that every pair of consecutive arcs $I_n$ and $I_{n+1}$ is joined by a pseudo-arc $P_{n}$. 
Let $C(k_1,k_2)=\bigcup_{n=k_1}^{k_2}I_n\cup P_{n}$ for extended integers $-\infty\leq k_1\leq k_2\leq \infty $. In particular $W=C(-\infty,\infty)$, and for any integers $k_1,k_2$ the continuum $C(k_1,k_2)$ is arc-like. Note that even though an arc cannot be mapped onto a pseudo-arc, the converse does not hold, see \cite{Fearnley}. In fact the pseudo-arc can be mapped onto any set homeomorphic to $C(k_1,k_2)$, where $k_1,k_2\in\mathbb{Z}$. This makes our considerations more delicate then for Slovak spaces in the proof of \cite[Lemma 6]{DST}.

If $G|_W$ is orientation reversing (i.e. $\lim_{r\to\pm\infty}G(r)=\mp\infty$), then since a one-to-one image of $W$ compactified by $-\infty$ and $\infty$ becomes an arc-like continuum, with $G$ extending to this compactification by setting  $G(-\infty)=\infty,G(\infty)=-\infty$,  and arc-like continua have the fixed point property \cite{Hamilton}, $G|_W$ must have a fixed point. Therefore we may assume that $G$ preserves orientation. Now note that if $G(I_n)\cap P_m\neq\emptyset $ for some $n,m\in\mathbb{Z}$ then $G(I_n)\cap P_m$ is a single point, as an arc cannot be mapped onto any nondegenerate subcontinuum of the pseudo-arc (the latter contains no locally connected continua). In addition, since $\{P_k\}_{k\in\mathbb{Z}}$ forms a null sequence we must have $\lim_{|k|\to\infty}\diam (G(P_k))=0$. Therefore there exists an $N\in\mathbb{N}$ such that for all $|k|>N$ there exists an integer $m_k$ such that $G(C(k,k))\subseteq B_{\epsilon_k}(C(m_k,m_k))$, where $\lim_{|n|\to\infty}\epsilon_k=0$ (where $B_{\epsilon_k}(C(m_k,m_k))$ denotes an $\epsilon_k$-neighborhood of $C(m_k,m_k)$). If there exists an $m_0\in \Z$ such that, either for all $k>N$ we have $m_k=k+m_0$, or for all $k<-N$ we have $m_k=k+m_0$, then $G$ agrees with $F^{m_0}$ on a dense set (this follows from the fact that arc-like continua have a coincidence point property \cite[Theorem 12.29]{Nadler}, and the fact that $\bigcup_{n=-\infty}^{-k-1} P_n$ as well as $\bigcup_{n=k+1}^\infty P_n$ are dense in $Y$), and so $G=F^{m_0}$, contradicting noninvertibility of $G$. Thus, there exists a positive integer $M$ such that for infinitely many $k$ ($|k|>M$) we must have $m_k=m_{k+1}$. Therefore there exists a $k_0>M$ such that $G(C(-k_0,k_0))\subset C(-k_0,k_0)$. Since $C(-k_0,k_0)$ is arc-like, it follows that $G$ has a fixed point in $C(-k_0,k_0)$ which contradicts minimality and the proof is complete.
\end{proof}
\begin{thm}\label{cor}
There exist minimal spaces without minimal noninvertible maps, that are neither $\mathbb{S}^1$ nor Slovak.
\end{thm}
\begin{proof}
This is a consequence of Theorem \ref{A} and Theorem \ref{hom}. 
\end{proof}	
\section{Minimal spaces with degenerate homeomorphism groups}\label{section:degenerate}
\begin{thm}\label{thm:degenerate}
	There exist uncountably many minimal spaces with degenerate homeomorphism groups. In addition, for any real number $r\geq 0$ there exists such a~space that admits a minimal map with topological entropy $r$. 
\end{thm}
\begin{proof}
	We modify the construction of the space $Y$ from \cite[Theorem 3.1]{BCO} (so the one described in the present paper in Subsection~\ref{subsec:AlmostSlovak} and adapted in the proof of Theorem \ref{A}) in the following way. 
	Let $\psi: [0,1]\to[0,1]$ be a piecewise linear map defined by:
	\begin{equation}\label{eq:pseudoarc+Knaster}
	\psi(x):= \begin{cases}
	3x & \text{ if } x\in [0,1/3];\\
	-3x+2 & \text{ if } x\in [1/3,2/3];\\
	3x-2 & \text{ if } x\in [2/3,1];
	\end{cases}
	\end{equation}
	Note that $K=\varprojlim ([0,1],\psi)$ is a $3$-fold Knaster continuum \cite{Rogers}, which has two endpoints, $(0,0,\ldots)$ and $(1,1,\ldots)$. Let $f:[0,1]\to [0,1]$ be the Henderson map from \cite{Hen}.
	We modify the function $g:[-1,2]\to[-1,2]$ from Subsection~\ref{subsec:AlmostSlovak} to $g':[-1,2]\to[-1,2]$ only on the subinterval $[0,1]$ (we rescale $\psi$ to be defined on the interval $[1/2,1]$):
	
	\begin{equation}\label{eq:pseudoarc+Knaster}
	g'|_{[0,1]}(x):= \begin{cases}
	f(x) & \text{ if } x\in [0,1/2];\\
	\psi(x) & \text{ if } x\in [1/2,1].
	\end{cases}
	\end{equation}
	Note that Theorem 3.1. from \cite{BCO} still holds with such a modification  and as an outcome we obtain a continuum $\mathcal{Z}$. Again, there is a special composant $W$ of $\mathcal{Z}$, that contains a family of nowhere locally connected continua $\{L_{n}\}_{n\in\mathbb{Z}}$ such that $L_n=K_n\cup P_n$, where $K_n$ and $P_n$ are the $3$-fold Knaster continuum and the pseudo-arc respectively, for each $n\in\mathbb{Z}$. Let $\{o_n\}=K_n\cap P_n$ for each $n\in\mathbb{Z}$. It follows from Theorem 3.1. in \cite{BCO} that $\mathcal{Z}$ is an almost Slovak space. In addition $\mathcal{Z}$ does not admit an orientation reversing homeomorphism, so $\mathcal{Z}$ is in fact a Slovak space that carries a~minimal homeomorphism $G$, generating the homeomorphism group of $\mathcal{Z}$. We will modify $\mathcal{Z}$ to obtain a space that admits no minimal homeomorphism. It follows again from Theorem 3.1 in \cite{BCO} that $G(o_n)=o_{n+1}$ for all $n\in\mathbb{Z}$. For every $n \in \mathbb{N}$ and every pseudo-arc $P_n \subset W$ we choose a proper subpseudo-arc  $p_n\subset P_n$ as follows.  Let $p_0\subset P_0$ be a pseudo-arc such that $p_0\cap K_0=\{o_0\}$ and $p_n=G^n(p_0)$ for each $n \geq 0$. Since the pseudo-arcs $\{p_n\}_{n \geq 0}$ form a null sequence, we can shrink each $p_n$ to the point $o_n$ ($n\geq 0$) to obtain a~new space $\mathcal{Z}'$ on which we have a map $G'$ which is semi-conjugate to the map $G$ on $\mathcal{Z}$, with semi-conjugacy given by the quotient map $\pi:\mathcal{Z}\to \mathcal{Z}'$, with $\pi(p_n)=o_n$ for all $n\geq 0$, see the diagram below.
	
	\begin{center}
		\begin{tikzpicture}\label{diagram:semiconjugacy}
		\matrix (m) [matrix of math nodes,row sep=3em,column sep=4em,minimum width=2em]
		{
			\mathcal{Z} & \mathcal{Z} \\
			\mathcal{Z}' & \mathcal{Z}' \\};
		\path[-stealth]
		(m-1-1) edge node [left] {$\pi$} (m-2-1)
		edge node [above] {$G$} (m-1-2)
		(m-2-1.east|-m-2-2) edge node [above] {$G'$} (m-2-2)
		(m-1-2) edge node [right] {$\pi$} (m-2-2);
		
		\end{tikzpicture}
		
	\end{center}
	
	Note that $G'$ is minimal, as a factor of the minimal homeomorphism $G$. However $G'$ is not invertible at $\pi(p_0)$. For each $n\in\mathbb{Z}$ we again have that $\pi(L_n)$ is a one-point union of the $3$-fold Knaster continuum $\pi(K_n)$ and the pseudo-arc $\pi(P_n)$, with the common point $\pi(o_n)=\pi(p_n)$ for each $n\geq 0$. The fact that $\pi(P_n)$, for each $n\geq 0$, is a pseudo-arc follows from Theorem 4 in \cite{B}, since $\pi(P_n)=P_n/p_n$ for all $n\geq 0$. Now suppose that $H:\mathcal{Z}'\to \mathcal{Z}'$ is a homeomorphism. Then  $H(\pi(K_n))=\pi(K_{n+k})$, $H(\pi(P_n))=\pi(P_{n+k})$ and so $H(\pi(o_n))=\pi(o_{n+k})$ for some $k\in\mathbb{Z}$. But then $\pi(G^k(o_n))=H(\pi(o_n))$ for all $n<0$. Since $\{\pi(o_n)\}_{n<0}$ is dense in $\mathcal{Z}'$  (because $(o_n)_{n<0}$ forms a dense orbit of a minimal homeomorphism $G$ in $Z$), and any map is uniquely determined on a dense set, it follows that $\pi\circ G^j=H\circ \pi$ for some $j\geq 0$. Consequently $H$ is noninvertible at $\pi(o_0)=\pi(p_0)$ unless $j=0$. This shows that $H=\operatorname{id}_{\mathcal{Z}'}$ and completes the proof. The construction yields uncountably many nonhomeomorphic examples, as we can start with any minimal suspension. For any real number $r\geq 0$, this minimal suspension can be chosen so that the topological entropy of the generating homeomorphism $G:\mathcal{Z}\to \mathcal{Z}$, and thus $G'$, is precisely $r$ (cf. Theorem 5 in \cite{DST}).
\end{proof}
\section{On indecomposability of Slovak spaces}\label{section:decomposable}
In this section we answer the following question from \cite{DST} in the negative. 

\begin{que}(Question 2 from \cite{DST}) Is every Slovak space an indecomposable continuum?
\end{que}

\begin{thm}\label{thm:decomposable}
	Each of the following classes of spaces contains an uncountable family of decomposable continua.
	\begin{itemize}
		\item[(D1)] Slovak spaces,
		\item[(D2)] almost Slovak spaces that are not Slovak,
		\item[(D3)] minimal spaces with degenerate homeomorphism groups.
	\end{itemize}
\end{thm}
\begin{proof}
	We start by proving (D1). Let $X$ be a~Slovak space obtained from an irrational flow by $\alpha$, and $Y$ be a~Slovak space obtained from an irrational flow by $\beta$ on a 2-adic solenoid, by the same method as $\mathcal{Z}$ in Section \ref{section:degenerate}, where $\alpha$ and $\beta$ are two irrationals that are rationally independent. In particular, both $X$ and $Y$ contain a special composant that contains a null sequence of the one-point unions of the 3-fold Knaster continuum and the pseudo-arc, which forms an orbit of the generating minimal homeomorphism. Let $W_X$ and $W_Y$ be the special composants of $X$ and $Y$ respectively, and $h_X,h_Y$ be the minimal generating homeomorphisms for the respective Slovak spaces. Consider the space $Z=X\times Y$ and note that it is a~decomposable continuum, as any product of continua is decomposable \cite{Jones}. Note also that $(h_X,h_Y)$ is a minimal homeomorphism of $Z$. However $Z$ is not Slovak, since it admits nonminimal homeomorphisms of the form $(\operatorname{id}_X,h^n_Y)$ and $(h^n_X,\operatorname{id}_Y)$ for $n\in\mathbb{Z}$. We are going to "kill" such homeomorphisms. 
	\begin{figure}[!htbp]
		\begin{center}
			\includegraphics[width=0.45\textwidth]{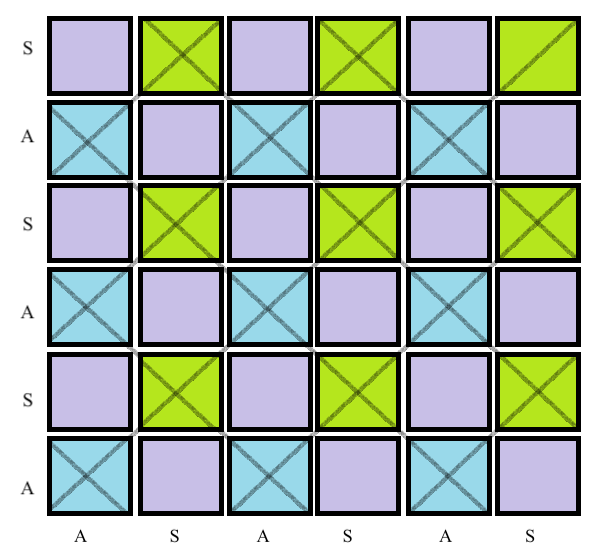}
			\includegraphics[width=0.45\textwidth]{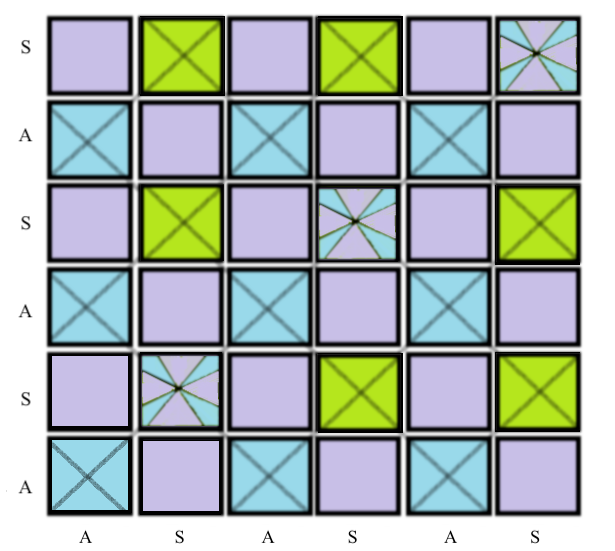}
			\begin{minipage}[c]{0.8\textwidth}
				\begin{center}
					\caption{The sets $W_X\times W_Y$ (left) and $\pi(W_X\times W_Y)$ (right), where $S$ is the one-point union of the 3-fold Knaster continuum and pseudo-arc, and $A$ is an arc.\label{pic:squares}}
				\end{center}
			\end{minipage}
		\end{center}
	\end{figure}
	Consider the subset $W_X\times W_Y$ of $Z$.  This set contains a "chessboard" of special "rectangles" $S_{n,m}$ that are Cartesian squares of the one-point union of the Knaster continuum and the pseudo-arc (there is a~natural correspondence with $\mathbb{Z}^2$ by considering the set of times of the initial flows $\{(n\alpha,m\beta): (n,m)\in\mathbb{Z}^2\}$). There are other "rectangles" that are either homeomorphic to a 2-disk, or to a Cartesian product of an arc with the one-point union of the Knaster continuum and pseudo-arc; see the left picture on Figure~\ref{pic:squares}. Let us consider the family $\mathcal{S}=\{S_{n,n}:n\in\mathbb{Z}\}$, of rectangles on the main diagonal, and an equivalence relation: 
	$$x\equiv y\textrm{ iff } x,y\in S_{n,n} \textrm{ for some } n\in\mathbb{Z}.$$
	Since $\mathcal{S}$ forms a null sequence, the quotient space $Z'=Z/_\equiv$ is a metric continuum, by Proposition 2 and Proposition 3 on p.13-14 in \cite{Daverman}. 
	
	It is easy to see that $Z'$ is decomposable, since if $A\cup B=Z$ with $A,B$ two distinct proper subcontinua of $Z$ then $\pi(A)\cup\pi(B)=Z'$, where $\pi:Z\to Z'$ is the quotient map. If $\pi(A)=\pi(B)$ then $((A\setminus B)\cup (B\setminus A))\subset \bigcup \mathcal{S}$, and since $Z\setminus \bigcup\mathcal{S}$ is dense it follows that $A$ and $B$ are dense in $Z$, and thus $A=B=Z$ which leads to a contradiction. We argue similarly if $\pi(A)\subset \pi(B)$ or $\pi(B)\subset \pi(A)$, since then $\pi(B)=Z'$ or $\pi(A)=Z'$ respectively. These contradictions show that $Z'$ is decomposable. 
	
	Note that the homeomorphism $H=\pi\circ (h_X,h_Y)\circ \pi^{-1}$ is minimal on $Z'$, as factors of minimal maps are minimal.  
	Now we shall show that $H$ is a generator of the homeomorphism group of $Z'$. By the same arguments as in the proof of Theorem 3.2 in \cite{BCO}, any homeomorphism of $Z$ is of the form $(h^\ell_X,h^j_Y)$, for some $\ell,j\in\mathbb{Z}$. Consequently, without loss of generality, we may assume that if $\ell\neq j$ then $j=0$ or $\ell=0$. Suppose by contradiction that there exists a homeomorphism $G:Z'\to Z'$ such that $G\neq H^k$ for all $k\in\mathbb{Z}$. Let $C_X\neq W_X$ be a composant of $X$. Then $G$ is a~translation on $\pi(C_X\times W_Y)$, since this set is homeomorphic with $C_X\times W_Y$ (by the fact that $\pi|(C_X\times W_Y)$ is a~homeomorphism).  Then $\pi\circ (\operatorname{id}_X,h^j_Y)(v)=G\circ \pi(v)$ for each $v\in C_X\times W_Y$ and some $j\in\mathbb{Z}$. But since $C_X\times W_Y$ is dense in $X\times Y$ we must have $\pi\circ (\operatorname{id}_X,h^j_Y)(v)=G\circ \pi(v)$ for all $v\in Z'$. Consequently $G=\pi\circ (\operatorname{id}_X,h^j_Y)\circ \pi^{-1}$. But since $(\operatorname{id}_X,h^j_Y)(S_{n,n})=S_{n,n+j}$, for the point $s=\pi(S_{n,n})$ and the nondegenerate continuum $S=\pi(S_{n,n+j})$ we get $G(s)=S$, unless $j=0$, a contradiction. Arguing the same way with $W_X\times C_Y$, where $C_Y\neq W_Y$ is a composant of $Y$ we get that $\ell=0$ and so $G=\operatorname{id}_{Z'}$ which is again a contradiction. So the homeomorphism group of $Z'$ is $\{H^k:k\in\mathbb{Z}\}$ and the proof is complete. Since we have uncountably many choices for $\alpha$ and $\beta$, there exist uncountably many spaces with the desired properties. 
	
	The proofs for (D2) and (D3) are analogous, by combining with the results of Section~\ref{sec:AlmostSlovakNotSlovak} and Section~\ref{section:degenerate}.
\end{proof}
\begin{rem}
	We could have constructed $X$ and $Y$ that contain a null sequence of $\sin(1/x)$-curves instead of Knaster continua with pseudo-arcs in the above proof, and use the same arguments thereafter. Starting with analogous Slovak spaces from \cite{DST} would also suffice.
\end{rem}
\section{Acknowledgements}
We are grateful to A. de Carvalho and P. Oprocha for helpful feedback on the results of the paper, and B. Vejnar for suggesting Question \ref{q1}. We are also grateful to L'. Snoha for many useful conversations on minimal spaces. This work was partially supported by subsidy for institutional development  IRP201824 ”Complex topological structures” from University of Ostrava and by the Faculty of Applied Mathematics AGH UST statutory tasks within subsidy of Polish Ministry of Science and Higher Education. 
J. Boro\'nski was also partially supported by National Science Centre, Poland (NCN), grant "Homogeneity and Minimality in Compact Spaces" no. 2015/19/D/ST1/01184. J. \v Cin\v c was also supported by the Austrian Science Fund (FWF) Schr\"odinger Fellowship stand-alone project J4276-N35.

\end{document}